\documentclass[12pt]{amsart}
\usepackage{amssymb}
\usepackage[all]{xy}

\newtheorem{theorem}{Theorem}[section]
\newtheorem{definition}[theorem]{Definition}
\newtheorem{proposition}[theorem]{Proposition}
\newtheorem{problem}[theorem]{Problem}

\begin{document}

\title[Quantum groups]{Quantum groups under very strong axioms}

\author{Teo Banica}
\address{T.B.: Department of Mathematics, University of Cergy-Pontoise, F-95000 Cergy-Pontoise, France. {\tt teo.banica@gmail.com}}


\begin{abstract}
We study the intermediate quantum groups $H_N\subset G\subset U_N^+$. The basic examples are $H_N,K_N,O_N,U_N,H_N^+,K_N^+,O_N^+,U_N^+$, which form a cube. Any other example $G$ sits inside the cube, and by using standard operations, namely intersection $\cap$ and generation $<\,,>$, can be projected on the faces and edges. We prove that under the strongest possible axioms, namely (1) easiness, (2) uniformity, and (3) geometric coherence of the various projection operations, the 8 basic solutions are the only ones.
\end{abstract}

\maketitle

\section*{Introduction}

The unitary group $U_N$ has a free analogue $U_N^+$, constructed by Wang in \cite{wa1}, and the study of the closed quantum subgroups $G\subset U_N^+$ is a problem of general interest. These subgroups fit into the general theory developed by Woronowicz in \cite{wo1}, \cite{wo2}, and can be studied by using representation theory methods. A generalization of the Peter-Weyl theory is available for them, an analogue of the Tannakian duality holds as well, and the Haar functional can be computed via a Weingarten integration formula. 

\bigskip

From the Tannakian perspective, the closed subgroups $G\subset U_N^+$ containing the usual symmetric group, $S_N\subset G$, are the ``simplest''. Indeed, at the level of the associated Tannakian categories we have a reverse inclusion, $C_G\subset C_{S_N}$, and when coupling this with the well-known fact that $C_{S_N}$ is a very elementary object, namely the span of the category of set-theoretic partitions $P$, we are led in this way into pure combinatorics.

Thus, at the core of the classification work for the compact quantum groups lies the question of classifying the intermediate subgroups $S_N\subset G\subset U_N^+$. The work so far here has focused on the ``easy'' case, where the Tannakian category of $G$ appears in the simplest possible way, namely $C_G=span(D)$, for a certain subcategory $D\subset P$. See \cite{bsp}.

\bigskip

From the point of view of the Drinfeld-Jimbo twisting \cite{dri}, \cite{jim}, the easy quantum group theory is a $q=1$ theory. Twisting it at $q\in\mathbb T$ is not possible in general, because Woronowicz's formalism, based on $C^*$-algebras, requires $q\in\mathbb R$. Thus, the reasonable problem is that of twisting the theory over $\mathbb T\cap\mathbb R=\{\pm1\}$, and so at $q=-1$.

In Tannakian terms, the $q=-1$ twisting requires introducing a signature function in the implementation of the partitions $\pi\in P$, as linear maps. Such a signature function exists indeed, but is defined only of the subcategory $P_{even}\subset P$ of partitions having even blocks. Now since we have $span(P_{even})=C_{H_N}$, with $H_N$ being the hyperoctahedral group, the conclusion is that the $q=-1$ twisting procedure requires $H_N\subset G$.

Summarizing, when taking into account both the easiness philosophy and the Drinfeld-Jimbo twisting philosophy, we are led to the assumption $H_N\subset G\subset U_N^+$.  See \cite{ba2}.

\bigskip

There are many examples of intermediate quantum groups $H_N\subset G\subset U_N^+$. Among them, there are 8 basic solutions, which form a cube, as follows:
$$
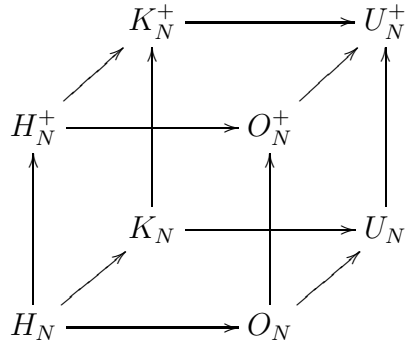
\xymatrix@R=20pt@C=20pt{
&K_N^+\ar[rr]&&U_N^+\\
H_N^+\ar[rr]\ar[ur]&&O_N^+\ar[ur]\\
&K_N\ar[rr]\ar[uu]&&U_N\ar[uu]\\
H_N\ar[uu]\ar[ur]\ar[rr]&&O_N\ar[uu]\ar[ur]
}$$

Here on the lower face we have the orthogonal and unitary groups $O_N,U_N$, and the hyperoctahedral group $H_N=S_N\wr\mathbb Z_2$ and its complex version $K_N=S_N\wr\mathbb T$. As for the quantum groups on the upper face, these are liberations, from \cite{bb+}, \cite{bbc}, \cite{wa1}.

\bigskip

Quite remarkably, the above cube is an ``intersection and generation'' diagram. In order to explain this property, consider any of the 6 faces of the cube, which is as follows: 
$$
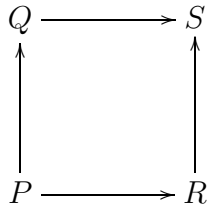
\xymatrix@R=50pt@C=50pt{
Q\ar[r]&S\\
P\ar[r]\ar[u]&R\ar[u]}$$

The point is that, given any such square diagram, we can use the intersection operation $\cap$, and the generation operation $<\,,>$, and impose the following conditions:
$$P=Q\cap R\quad,\quad<Q,R>=S$$

Now back to our cube, one can prove that this is indeed an intersection and generation diagrams, as a consequence of the various results in \cite{bb+}, \cite{bbc}, \cite{bsp}.

\bigskip

With this observation in hand, let us go back to the $H_N\subset G\subset U_N^+$ problem. Any solution $G$ sits inside the cube, and by using the operations $\cap$ and $<\,,>$, we can ``project'' this quantum group on the various faces and edges of the cube. Under suitable assumptions, we end up in this way with a ``slicing'' of the cube, into 8 small cubes.

To be more precise, we can first associate to $G$ six quantum groups, as follows:
$$G_{class}=G\cap U_N\quad,\quad G_{free}=<G,H_N^+>$$
$$G_{disc}=G\cap K_N^+\quad,\quad G_{cont}=<G,O_N>$$
$$G_{real}=G\cap O_N^+\quad,\quad G_{unit}=<G,K_N>$$

We have then inclusions between $G$ and these quantum groups, as follows:
$$
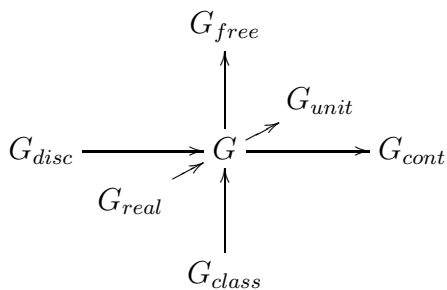
\xymatrix@R=2pt@C=2pt{
&&G_{free}&&\\
\\
&&&G_{unit}&\\
G_{disc}\ar[rr]&&G\ar[rr]\ar[uuu]\ar[ur]&&G_{cont}\\
&G_{real}\ar[ur]&&&\\
\\
&&G_{class}\ar[uuu]&&
}$$

This diagram can be fit inside the original cube, in the obvious way. Moreover, under suitable compatibility assumptions between the above operations, we can project on the edges as well, and we end up with a slicing of the original cube, into 8 small cubes.

Since we have a diagram formed by square subdiagrams, we can formulate:

\bigskip

\noindent {\bf Definition.} {\em We say that an intermediate subgroup $H_N\subset G\subset U_N^+$ has the slicing property if the cube slicing that it produces is an intersection and generation diagram.}

\bigskip

In order to complete our study, we will need one more concept. We recall that a family of quantum groups $G=(G_N)$ with $G_N\subset U_N^+$ is called uniform when the following conditions are satisfied, with respect to the standard embeddings $U_{N-1}^+\subset U_N^+$:
$$G_{N-1}=G_N\cap U_{N-1}^+$$

This is something very natural, algebrically speaking. At a more advanced level, this condition appeared in \cite{bsp}, in connection with the Bercovici-Pata bijection \cite{bpa}, and also in \cite{ba1}, \cite{bss}, in connection with various noncommutative geometry questions.

With these preliminaries in hand, we can now formulate:

\bigskip

\noindent {\bf Theorem.} {\em Assume that $H_N\subset G\subset U_N^+$ has the following properties:
\begin{enumerate}
\item Easiness.

\item Uniformity.

\item Slicing property.
\end{enumerate}
Then $G$ must be one of the basic $8$ quantum groups.}

\bigskip

This will be our main result. Of course, this is quite philosophical. We will explain as well on how to ``build'' on this result, by removing or modifying some of the axioms.

\bigskip

The paper is organized as follows: in 1 we recall the construction of the main cube, in 2 we explain in detail the above slicing procedure, in 3 we prove the above theorem, and in 4 we briefly discuss further classification results, along these lines.

\section{The cube}

We use Woronowicz's quantum group formalism from \cite{wo1}, \cite{wo2}, under the extra assumption $S^2=id$. To be more precise, the definition that we will need is:

\begin{definition}
Assume that $(A,u)$ is a pair consisting of a unital $C^*$-algebra $A$, and a unitary matrix $u\in M_N(A)$ whose coefficients generate $A$, such that the formulae
$$\Delta(u_{ij})=\sum_ku_{ik}\otimes u_{kj}\quad,\quad \varepsilon(u_{ij})=\delta_{ij}\quad,\quad S(u_{ij})=u_{ji}^*$$
define morphisms of $C^*$-algebras $\Delta:A\to A\otimes A$, $\varepsilon:A\to\mathbb C$, $S:A\to A^{opp}$. We write then $A=C(G)$, and call $G$ a compact matrix quantum group.
\end{definition}

The basic examples are the compact Lie groups, $G\subset U_N$. Indeed, given such a group we can set $A=C(G)$, and let $u_{ij}:G\to\mathbb C$ be the standard coordinates, $u_{ij}(g)=g_{ij}$. The axioms are then satisfied, with $\Delta,\varepsilon,S$ being the functional analytic transposes of the multiplication $m:G\times G\to G$, unit map $u:\{.\}\to G$, and inverse map $i:G\to G$.

Another class of examples is provided by the abstract duals $G=\widehat{\Gamma}$ of the finitely generated discrete groups $\Gamma=<g_1,\ldots,g_N>$. Indeed, we can set $A=C^*(\Gamma)$, and let $u=diag(g_1,\ldots,g_N)$ be the diagonal matrix formed by the generators. Once again the axioms are satisfied, and when $\Gamma$ is abelian we have an identification $A=C(G)$.

We are particularly interested here in the orthogonal group $O_N$, the unitary group $U_N$, the hyperoctahedral group $H_N=S_N\wr\mathbb Z_2$, and its complex version $K_N=S_N\wr\mathbb T$. These groups have free analogues, constructed in \cite{bb+}, \cite{bbc}, \cite{wa1}, as follows:

\begin{proposition}
We have compact quantum groups, whose associated algebras are constructed by starting with an abstract $N\times N$ matrix $u=(u_{ij})$, as follows,
\begin{enumerate}
\item $U_N^+$, obtained by imposing the conditions $u^*=u^{-1},u^t=\bar{u}^{-1}$,

\item $O_N^+\subset U_N^+$, obtained by further imposing the conditions $u=\bar{u}$,

\item $K_N^+\subset U_N^+$, obtained via the conditions $u_{ij}^*u_{ij}=u_{ij}u_{ij}^*=p_{ij}=$ magic,

\item $H_N^+$, obtained by imposing the conditions for both $O_N^+$ and $K_N^+$,
\end{enumerate}
with the magic condition stating that $p_{ij}$ are projections, with $\sum_ip_{ij}=\sum_jp_{ij}=1$.
\end{proposition}

\begin{proof}
All this is standard, the idea in each case being that if $u=(u_{ij})$ satisfies the conditions, then so do the matrices $u^\Delta=(\sum_ku_{ik}\otimes u_{kj})$, $u^\varepsilon=(\delta_{ij})$, $u^S=(u_{ji}^*)$. Thus we can construct $\Delta,\varepsilon,S$ as in Definition 1.1, by using the universal property of $C(G)$.
\end{proof}

There are many things known about the above 4 quantum groups, in analogy with the known results about their classical counterparts. In the continuous case the passage $O_N\to U_N$ is best understood as a complexification at the Lie algebra level, and the free analogue of this fact states that $O_N^+\to U_N^+$ is a ``free complexification'', in a certain algebraic geometry sense. In the discrete case, the key identifications $H_N=S_N\wr\mathbb Z_2$ and $K_N=S_N\wr\mathbb T$ have free counterparts $H_N^+=S_N^+\wr_*\mathbb Z_2$ and $K_N^+=S_N^+\wr_*\mathbb T$, with $S_N^+$ being Wang's quantum permutation group \cite{wa2}, and with $\wr_*$ being Bichon's free wreath product operation \cite{bic}. For more on these topics, we refer to \cite{bb+}, \cite{bbc}.

In order to study these 4 quantum groups, and other quantum groups of the same type, we will need Woronowicz's Tannakian duality result from \cite{wo2}, in its ``soft'' form, worked out by Malacarne in \cite{mal}. The precise statement that we need is as follows:

\begin{proposition}
The closed subgroups $G\subset U_N^+$ are in correspondence with their Tannakian categories $C(k,l)=Hom(u^{\otimes k},u^{\otimes l})$, the correspondence being given by
$$C(G)=C(U_N^+)\Big/\left<T\in Hom(u^{\otimes k},u^{\otimes l})\Big|\forall k,l,\forall T\in C(k,l)\right>$$ 
where all the exponents are by definition colored integers, with the corresponding tensor powers being defined by $u^{\otimes\emptyset}=1,u^{\otimes\circ}=u,u^{\otimes\bullet}=\bar{u}$ and multiplicativity.
\end{proposition}

\begin{proof}
As already mentioned, this result is from \cite{mal}. The idea is that we have a surjective arrow from left to right, and the injectivity can be checked by doing some algebra, and then by applying the bicommutant theorem, as a main tool. See \cite{mal}.
\end{proof}

As a last ingredient, we will need the definition and basic properties of the $\cap$ and $<\,,>$ operations for the closed subgroups of $U_N^+$. We can proceed here as follows:

\begin{proposition}
The closed subgroups of $U_N^+$ are subject to $\cap$ and $<\,,>$ operations, constucted via the above Tannakian correspondence $G\to C_G$, as follows:
\begin{enumerate}
\item Intersection: defined via $C_{G\cap H}=<C_G,C_H>$.

\item Generation: defined via $C_{<G,H>}=C_G\cap C_H$.
\end{enumerate}
In the classical case, where $G,H\subset U_N$, we obtain in this way the usual notions.
\end{proposition}

\begin{proof}
Since the $\cap$ and $<\,,>$ operations are clearly well-defined for the Tannakian categories, the operations in (1,2) make sense indeed. As for the last assertion, this is something well-known, which follows from definitions, via an elementary computation.
\end{proof}

The above statement is of course something quite compact. It is possible to develop some more theory, with universality diagrams, other abstract aspects, and more examples as well. We refer here to \cite{bcf}, where these operations are heavily used.

With these ingredients in hand, we can go back now to the basic 4 groups and 4 quantum groups, and formulate a key result about them, as follows:

\begin{theorem}
The basic quantum unitary and quantum reflection groups, with the inclusions between them, form a cubic diagram, as follows:
$$\xymatrix@R=20pt@C=20pt{
&K_N^+\ar[rr]&&U_N^+\\
H_N^+\ar[rr]\ar[ur]&&O_N^+\ar[ur]\\
&K_N\ar[rr]\ar[uu]&&U_N\ar[uu]\\
H_N\ar[uu]\ar[ur]\ar[rr]&&O_N\ar[uu]\ar[ur]
}$$
Moreover, this is an intersection/generation diagram, in the sense that for any of its square subdiagrams $P\subset Q,R\subset S$ we have $P=Q\cap R$ and $<Q,R>=S$.
\end{theorem}

\begin{proof}
The first assertion is clear from definitions. In order to prove now the second assertion, we must compute the Tannakian categories of our 8 quantum groups. 

For this purpose, we use the easy quantum group philosophy. Let us recall indeed from \cite{bsp} that associated to any partition $\pi\in P(k,l)$ between an upper row of $k$ points and a lower row of $l$ points is the following linear map, between tensor powers of $\mathbb C^N$:
$$T_\pi(e_{i_1}\otimes\ldots\otimes e_{i_k})=\sum_{j_1,\ldots,j_l}\delta_\pi\begin{pmatrix}i_1&\ldots&i_k\\ j_1&\ldots&j_l\end{pmatrix}e_{j_1}\otimes\ldots\otimes e_{j_l}$$

Here $\delta_\pi\in\{0,1\}$ is a Kronecker type symbol, whose value depends on whether the indices fit or not, when put in the obvious way on the legs of the partition. See \cite{bsp}.

With this construction in hand, the result regarding our 8 quantum groups, which is something well-known, from Brauer's paper \cite{bra}, and then from a number of extensions, including \cite{bb+}, \cite{bbc}, \cite{bsp}, is that we obtain categories of the following type: 
$$C(k,l)=span\left(T_\pi\Big|\pi\in D\right)$$

To be more precise, let $P_{even}$ be the set of partitions all whose blocks have even size, let $P_2\subset P_{even}$ be the set of pairings, let $\mathcal P_{even}\subset P_{even}$ be the set of partitions which are ``matching'', in the sense that we have $\#\circ=\#\bullet$ in each block, when counting the upper legs with sign $-$ and the lower legs with sign $+$, and finally let $NC_{even}\subset P_{even}$ be the set of partitions which are noncrossing. We have so far 4 sets of partitions, and by further intersecting these sets we obtain 4 more sets, denoted $P_2,NC_2,\mathcal{NC}_{even},\mathcal{NC}_2$.

Observe that these $4+4$ sets of partitions are ``categories of partitions'' in the sense of \cite{bsp}, in the sense that they are stable under the vertical and the horizontal concatenation of the partitions, and under the upside-down turning operation.

With these conventions, the above-mentioned result states that the quantum groups in the statement appear respectively from the following categories of partitions:
$$
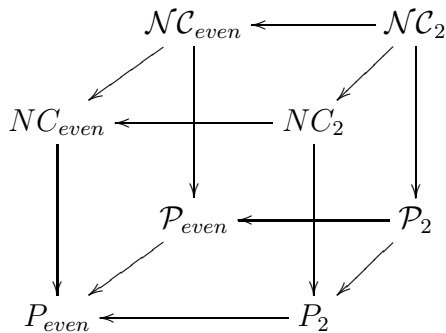
\xymatrix@R=20pt@C8pt{
&\mathcal{NC}_{even}\ar[dl]\ar[dd]&&\mathcal {NC}_2\ar[dl]\ar[ll]\ar[dd]\\
NC_{even}\ar[dd]&&NC_2\ar[dd]\ar[ll]\\
&\mathcal P_{even}\ar[dl]&&\mathcal P_2\ar[dl]\ar[ll]\\
P_{even}&&P_2\ar[ll]
}$$

Now observe that this diagram of categories of partitions is both an intersection and generation diagram. Indeed, both these properties are elementary to check.

When getting back now to the quantum groups, via the Tannakian duality operation, $D\to G_D$, it follows from Proposition 1.4 that we have the following formulae:
$$G_{D\cap E}=<G_D,G_E>\quad,\quad G_{<D,E>}=G_D\cap G_E$$

Thus, our intersection and generation diagram of categories of partitions gets transformed into an intersection and generation diagram of quantum groups, as stated.
\end{proof}

We will heavily use in what follows the above result, as well as the various technical ingredients developed in its proof. In addition to what has been said there, let us mention that the quantum groups whose Tannakian categories are of the form $C=span(D)$, as in the above proof, are called ``easy''. For more on easiness, we refer to \cite{bsp}, \cite{csn}.

\section{Cube slicing}

We are interested in the classification of the intermediate subgroups $H_N\subset G\subset U_N^+$. Such a quantum group can be imagined as sitting inside the cube, and the point is that by using the operations $\cap$ and $<\,,>$, we can ``project'' it on the faces and edges.

In order to clarify this construction, let us start with the following definition:

\begin{definition}
Associated to any quantum group $H_N\subset G\subset U_N^+$ are:
\begin{enumerate}
\item Its classical version, $G_{class}=G\cap U_N$.

\item Its free version, $G_{free}=<G,H_N^+>$.

\item Its discrete version, $G_{disc}=G\cap K_N^+$.

\item Its continuous version, $G_{cont}=<G,O_N>$.

\item Its real version, $G_{real}=G\cap O_N^+$.

\item Its unitary version, $G_{unit}=<G,K_N>$.
\end{enumerate}
\end{definition}

Here we have chosen the word ``unitary'' instead of the word ``complex'' in order for the corresponding abbreviation ``unit'' not to be confused with ``cont''.

In relation with our cube, we can now formulate:

\begin{proposition}
Given an intermediate quantum group $H_N\subset G\subset U_N^+$, we have a diagram of closed subgroups of $U_N^+$, obtained by inserting
$$
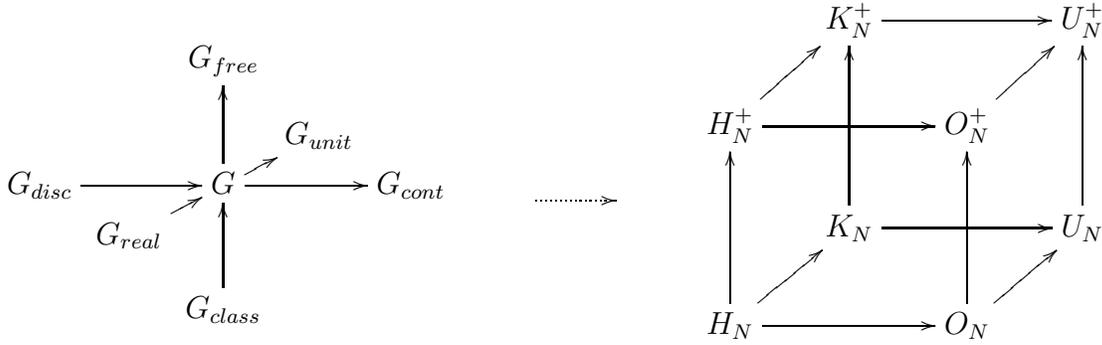
\xymatrix@R=2pt@C=2pt{
\\
&&G_{free}&&\\
\\
&&&G_{unit}&\\
G_{disc}\ar[rr]&&G\ar[rr]\ar[uuu]\ar[ur]&&G_{cont}\\
&G_{real}\ar[ur]&&&\\
\\
&&G_{class}\ar[uuu]&&}
\qquad\xymatrix@R=10pt@C=30pt{\\ \\ \\ \\ \ar@.[r]&}\qquad
\xymatrix@R=20pt@C=20pt{
&K_N^+\ar[rr]&&U_N^+\\
H_N^+\ar[rr]\ar[ur]&&O_N^+\ar[ur]\\
&K_N\ar[rr]\ar[uu]&&U_N\ar[uu]\\
H_N\ar[uu]\ar[ur]\ar[rr]&&O_N\ar[uu]\ar[ur]
}$$
in the obvious way, with each $G_x$ belonging to the main diagonal of each face.
\end{proposition}

\begin{proof}
The fact that we have indeed the diagram of inclusions on the left is clear from Definition 2.1 above. Regarding now the insertion procedure, consider any of the faces of the cube, $P\subset Q,R\subset S$. Our claim is that the corresponding quantum group $G_x$ can be inserted on the corresponding main diagonal $P\subset S$, as follows:
$$
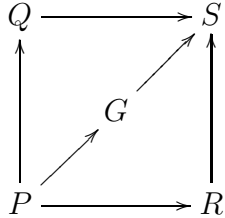
\xymatrix@R=20pt@C=20pt{
Q\ar[rr]&&S\\
&G\ar[ur]\\
P\ar[rr]\ar[uu]\ar[ur]&&R\ar[uu]}$$

Thus, in order to finish, we have to check a total of $6\times 2=12$ inclusions. But, according to Definition 2.1 above, these inclusions to be checked are as follows:

(1) $H_N\subset G_{class}\subset U_N$, where $G_{class}=G\cap U_N$.

(2) $H_N^+\subset G_{free}\subset U_N^+$, where $G_{free}=<G,H_N^+>$.

(3) $H_N\subset G_{disc}\subset K_N^+$, where $G_{disc}=G\cap K_N^+$.

(4) $O_N\subset G_{cont}\subset U_N^+$, where $G_{cont}=<G,O_N>$.

(5) $H_N\subset G_{real}\subset O_N^+$, where $G_{real}=G\cap O_N^+$.

(6) $K_N\subset G_{unit}\subset U_N^+$, where $G_{unit}=<G,K_N>$.

All these statements being trivial from the definition of $\cap$ and $<\,,>$, and from our assumption $H_N\subset G\subset U_N^+$, our insertion procedure works indeed, and we are done.
\end{proof}

In order now to complete the diagram, we have to project as well $G$ on the edges of the cube. For this purpose, we can assume that $G$ lies on one of the six faces.

The general result that we will need is as follows:

\begin{proposition}
Given an intersection and generation diagram $P\subset Q,R\subset S$ and an intermediate quantum group $P\subset G\subset S$, we have a diagram as follows:
$$\xymatrix@R=30pt@C=30pt{
Q\ar[r]&<G,Q>\ar[r]&S\\
G\cap Q\ar[u]\ar[r]&G\ar[r]\ar[u]&<G,R>\ar[u]\\
P\ar[r]\ar[u]&G\cap R\ar[u]\ar[r]&R\ar[u]}$$
In addition, $G$ slices the square, in the sense that this is an intersection and generation diagram, precisely when $G=<G\cap Q,G\cap R>$ and $G=<G,Q>\cap<G,R>$.
\end{proposition}

\begin{proof}
This is indeed clear from definitions, because the intersection and generation diagram conditions are automatic for the upper left and lower right squares, as well as half of the generation diagram conditions for the lower left and upper right squares.
\end{proof}

Let us record as well the Tannakian version of this construction, as follows:

\begin{proposition}
Given an intersection and generation diagram $P\subset Q,R\subset S$ and an intermediate quantum group $P\subset G\subset S$, we have a diagram as follows:
$$\xymatrix@R=30pt@C=30pt{
C_Q\ar[d]&C_G\cap C_Q\ar[l]\ar[d]&C_S\ar[l]\ar[d]\\
<C_G,C_Q>\ar[d]&C_G\ar[l]\ar[d]&C_G\cap C_R\ar[d]\ar[l]\\
C_P&<C_G,C_R>\ar[l]&C_R\ar[l]}$$
Also, $C_G$ slices the square, in the sense that this is an intersection and generation diagram, precisely when $C_P=<C_G,C_Q>\cap<C_G,C_R>$ and $C_S=<C_G\cap C_Q,C_G\cap C_R>$.
\end{proposition}

\begin{proof}
This is indeed clear from definitions, by proceeding as in the proof of Proposition 2.3. Observe that this is the diagram of Tannakian categories for the quantum groups in Proposition 2.3, due to the conversion formulae from Proposition 1.4 above.
\end{proof}

Now back to the cube, given an intermediate subgroup $H_N\subset G\subset U_N^+$ we can perform the construction in Proposition 2.3, for each of the 6 faces of the cube. The problem, however, is that we will not obtain a diagram of inclusions in this way, because for each of the 12 edges of the cube, the corresponding ``midpoint'' will be defined twice. 

As a first observation, 6 of these midpoints are actually well-defined, and there is no problem with them, because we have some compatibility formulae, as follows:

\begin{proposition}
We have the following results:
\begin{enumerate}
\item $(G_{class})_{disc}=(G_{disc})_{class}=G\cap K_N$. 

\item $(G_{class})_{real}=(G_{real})_{class}=G\cap O_N$.

\item $(G_{disc})_{real}=(G_{real})_{disc}=G\cap H_N^+$.

\item $(G_{free})_{cont}=(G_{cont})_{free}=<G,O_N^+>$.

\item $(G_{free})_{unit}=(G_{unit})_{free}=<G,K_N^+>$.

\item $(G_{cont})_{unit}=(G_{unit})_{cont}=<G,U_N>$.
\end{enumerate}
\end{proposition}

\begin{proof}
The formulae (1,2,3) in the statement are all of the following type:
$$(G\cap Q)\cap R=(G\cap R)\cap Q=G\cap P$$

As for the formulae (4,5,6), these are all of the following type:
$$<<G,Q>,R>=<<G,R>,Q>=<G,S>$$

Thus the value of $G$ is in fact irrelevant, and the results simply follow from the fact that the cube is an intersection and generation diagram.
\end{proof}

Regarding now the remaining 6 edges, the compatibility conditions here for the midpoints are not automatic, and we have to introduce the following notion:

\begin{definition}
We say that $G$ pre-slices the cube if it satisfies the following conditions:
\begin{enumerate}
\item $(G_{class})_{cont}=(G_{cont})_{class}$.

\item $(G_{class})_{unit}=(G_{unit})_{class}$.

\item $(G_{disc})_{free}=(G_{free})_{disc}$.

\item $(G_{disc})_{unit}=(G_{unit})_{disc}$.

\item $(G_{real})_{free}=(G_{free})_{real}$. 

\item $(G_{real})_{cont}=(G_{cont})_{real}$.
\end{enumerate}
\end{definition}

In other words, we are asking here for each intersection operation in Definition 2.1 to commute with the generation operations, except for the ``opposite'' operation. 

We can now formulate our first slicing result, as follows:

\begin{proposition}
Assuming that $G$ pre-slices the cube in the above sense, the diagram in Proposition 2.2 can be completed, via the construction in Proposition 2.3, into a diagram fully slicing the cube along the $3$ coordinates axes, into $8$ small cubes.
\end{proposition}

\begin{proof}
As already mentioned, checking that the conclusion holds is a matter of checking that the 12 projections on the edges are well-defined. And the situation is as follows:

(1) Regarding the $3$ edges emanating from $H_N$, and the $3$ edges landing into $U_N^+$, the result here follows from the formulae in Proposition 2.5 above.

(2) Regarding the remaining $6$ edges, not emanating from $H_N$ or landing into $U_N^+$, here the result follows from the formulae in Definition 2.6 above.
\end{proof}

We are not done yet with the slicing work, because nothing guarantees that our slicing is ``neat'', in the sense that we obtain in this way an intersection and generation diagram. In order to have this property, we have to introduce one more definition, as follows:

\begin{definition}
We say that $G$ slices the cube when
\begin{enumerate}
\item $G$ slices the cube, in the sense of Definition 2.6 above,

\item $G_{class},G,G_{free}$ slice the classical/intermediate/free faces,

\item $G_{disc},G,G_{cont}$ slice the discrete/intermediate/continuous faces,

\item $G_{real},G,G_{unit}$ slice the real/intermediate/unitary faces,
\end{enumerate}
where by ``intermediate'' we mean in each case ``parallel to its neighbors''.
\end{definition}

In short, we are asking here for a total of $6\times4=24$ conditions to be satisfied, namely the $6$ slicing conditions from Definition 2.6, and then the $6\times 3=18$ conditions coming from the 2 conditions in Proposition 2.3, applied to the $3\times3$ faces of the slicing.

We can now finish the slicing procedure, as follows:

\begin{theorem}
Assuming that $G$ slices the cube in the above sense, we have a diagram fully slicing the cube, into $8$ small cubes, which is an intersection diagram, in the sense that each of its $36$ small square faces is an intersection and generation diagram.
\end{theorem}

\begin{proof}
Here the first assertion follows from Proposition 2.7 above, and the second assertion follows from Proposition 2.3 above, via our assumptions from Definition 2.8.
\end{proof}

Summarizing, we know now how to slice the cube, with the remark that the $36\times2=72$ intersection and generation properties that are a priori needed, for the various small faces, collapse in fact to the $6\times4=24$ conditions from Definition 2.8 above.

Of course, this $72\to24$ simplification is something trivial, simply coming from the fact that the cube is an intersection and generation diagram, and not involving $G$ itself. It is probably possible to do here slightly better, by carefully examining the situation. 

However, when going on this way, at some point there will be certainly need for a deep result. So, the problem of fully simplifying our axioms is open, and interesting.

\section{Classification}

All the above is quite natural, and looking for quantum groups having the slicing property is an interesting question. However, our purpose here is a bit different. We are interested in formulating a foundational result, rather than something technical.

In order to do so, we must introduce one more concept, as follows:

\begin{definition}
A family of compact quantum groups $G=(G_N)$, with $G_N\subset U_N^+$ for any $N\in\mathbb N$, is called uniform when the following conditions are satisfied,
$$G_{N-1}=G_N\cap U_{N-1}^+$$
with respect to the standard embeddings $U_{N-1}^+\subset U_N^+$, given by $u=diag(1,v)$.
\end{definition}

This condition is something very natural, algebrically speaking, because we are here in an injective/projective limit situation for the associated compact and discrete quantum groups. At a more advanced level, this condition appeared in \cite{bsp}, in connection with the Bercovici-Pata bijection \cite{bpa} for the asymptotic laws of truncated characters, and also in \cite{ba1}, \cite{bss}, in connection with various noncommutative geometry questions. See \cite{ba3}.

We can now prove the statement announced in the introduction, namely:

\begin{theorem}
Assume that $H_N\subset G\subset U_N^+$ has the following properties:
\begin{enumerate}
\item Easiness.

\item Uniformity.

\item Slicing property.
\end{enumerate}
Then $G$ must be one of the basic $8$ quantum groups.
\end{theorem}

\begin{proof}
The idea will be that of ``locating'' our quantum group inside the cube, in a 3D sense, by using the slicing property. There are many ways in doing so, and in view of the known classification results, whose technical level can vary a lot, the best is by using the results for the classical face from \cite{twe}, and the results for the orthogonal edge from \cite{bve}. In other words, we would like to use the ``coordinate system'' highlighted below:
$$\xymatrix@R=18pt@C=18pt{
&K_N^+\ar[rr]&&U_N^+\\
H_N^+\ar[rr]\ar[ur]&&O_N^+\ar[ur]\\
&K_N\ar@=[rr]\ar[uu]&&U_N\ar[uu]\\
H_N\ar[uu]\ar@=[ur]\ar@=[rr]&&O_N\ar@=[uu]\ar@=[ur]
}$$

Let us start with the classical face. Our goal here is that of finding the possible values of $G_{class}$, which belong by definition to this face. In order to simplify the discussion, we will temporarily assume $G=G_{class}$. Thus, we would like to find the intermediate quantum groups $H_N\subset G\subset U_N$ which are easy, uniform, and which slice the lower face.

According to Proposition 2.3, the slicing diagram for the lower face is as follows:
$$\xymatrix@R=30pt@C=30pt{
K_N\ar[r]&G_{unit}\ar[r]&U_N\\
G_{disc}\ar[u]\ar[r]&G\ar[r]\ar[u]&G_{cont}\ar[u]\\
H_N\ar[r]\ar[u]&G_{real}\ar[u]\ar[r]&O_N\ar[u]}$$

With these preliminaries in hand, we can now survey the known results on the subject. There are several statements here, all based on \cite{twe}, as follows:

\bigskip

\underline{Classical face, easy case.} The full classification of the intermediate easy quantum groups $H_N\subset G\subset U_N$ is available from \cite{twe}. There are many examples here, with the whole subject being quite technical, and we refer to \cite{twe} for the full details.

\bigskip

\underline{Classical face, easy uniform case.} As explained in \cite{ba1}, in the context of the noncommutative homogeneous space considerations there, which require the uniformity axiom, imposing this axiom leads to some simplifications, the solutions being as follows:
$$\xymatrix@R=23pt@C=35pt{
K_N\ar[rr]&&U_N\\
H_N^s\ar[u]&&\ar[u]\\
H_N\ar[rr]\ar[u]&&O_N\ar[uu]}$$ 

Here the extra groups on the left are the complex reflection groups $H_N^s=\mathbb Z_s\wr S_N$ with $s\in\{2,4,6,\ldots,\infty\}$ from \cite{bb+}, which at $s=2,\infty$ cover $H_N,K_N$. See \cite{ba1}. 

\bigskip

\underline{Classical face, easy slicing case.} As explained in \cite{ba3}, when imposing the slicing condition on the lower face, which comes from the general noncommutative geometry considerations in \cite{ba2}, \cite{bbi}, some simplifications appear as well, the solutions being as follows:
$$\xymatrix@R=25pt@C=30pt{
K_N\ar[rr]&&U_N\\
H_{N,L}\ar[u]\ar[rr]&&O_{N,L}\ar[u]\\
H_N\ar[rr]\ar[u]&&O_N\ar[u]}$$ 

Here the various extra groups are obtained by ``arithmetic complexification'', according to the formula $G_L=\mathbb Z_LG$, with $L\in\{2,3,\ldots,\infty\}$. See \cite{ba2}, \cite{ba3}, \cite{bbi}.

\bigskip

\underline{Classical face, easy uniform slicing case.} In order to obtain the solutions here, we just have to intersect the above two diagrams, and we obtain as solutions:
$$\xymatrix@R=25pt@C=30pt{
K_N\ar[rr]&&U_N\\
\\
H_N\ar[rr]\ar[uu]&&O_N\ar[uu]}$$ 

In short, getting back now to our original problem, we have reached to the conclusion that $G_{class}$ must be one of the 4 vertices of the lower face of the cube.

\bigskip

With this result in hand, we can now go ahead, and finish by using \cite{bve}. Indeed, the projection of $G$ on the real continuous edge, $O_N\subset O_N^+$, must be an intermediate easy quantum group $O_N\subset G\subset O_N^+$. But, according to \cite{bve}, the only non-trivial solution here is the half-classical orthogonal group $O_N^*$, coming from the half-commutation relations $abc=cba$. And this quantum group being not uniform, simply because $abc=cba$ with $c=1$ imply $ab=ba$, as explained in \cite{bss}, we have only $O_N,O_N^+$ as solutions. 

\bigskip

Summarizing, our intermediate quantum group $H_N\subset G\subset U_N^+$ must lie on the upper or the lower face of the cube, and its projection on the lower face must be one of the 4 vertices of the lower face. Thus $G$ must be one of the 8 vertices of the cube, as claimed.
\end{proof}

\section{Conclusion}

We have seen that by ``piling up'' a number of axioms, which are natural in the noncommutative geometry and free probability context, and which actually came from a substantial amount of work in this direction, we are left with 8 quantum groups. 

All this is of course quite philosophical. What we have here is rather some kind of ``ground zero'' result, proving a foundational framework for more specialized classification results, which can be obtained by carefully modifying of removing the axioms.

Here is a brief discussion, regarding the modification/removal of these axioms:

\begin{problem}
Modifying or removing the slicing axiom.
\end{problem}

Generally speaking, totally removing the slicing axiom leads into some difficult questions, with the problem coming from a lack of 3D orientation inside the cube.

One fruitful direction, however, comes by restricting the attention to the 6 faces of the cube, and trying to find the uniform easy quantum groups which slice the face.

Skipping the details here, let us mention that the problem is solved by \cite{twe} for the upper and lower faces, is elementary as well for the front and right face, using the results from \cite{bve}, \cite{twe}, and is still in need of some non-trivial combinatorial work, based on the results in \cite{mw1}, \cite{mw2}, \cite{rwe}, in what regards the left face and the bottom face.

\begin{problem}
Modifying or removing the uniformity axiom.
\end{problem}

This is another interesting direction. The general strategy from the proof of Theorem 3.2 above can be followed, with the only piece of work still needed being that of unifying the constructions on $H_{N,L},O_{N,L}$ with the construction of the half-liberations.

\begin{problem}
Modifying or removing the easiness axiom.
\end{problem}

This is something heavier. In principle the general strategy from the proof of Theorem 3.2 above can be followed too, the work in the classical case being probably something quite standard, and with conjectural input for the orthogonal edge coming from \cite{bbs}.

There are as well several modifications of the easy quantum group theory which can be used, the most standard ones, at least for now, coming from the work in \cite{bsk}, \cite{fr1}, \cite{fr2}.

\begin{problem}
Modifying or removing the $H_N\subset G$ axiom.
\end{problem}

Once again, this something heavier. A natural direction, which would require however rethinking the slicing procedure, is that of using the condition $S_N\subset G$. Note that this would require as well in dealing with all the ``singleton issues'' which might appear.

Alond the same lines, using some even weaker conditions, of type $A_N\subset G$, make sense as well, at least theoretically. For some comments here, we refer to \cite{ba4}.

\begin{problem}
Modifying everything.
\end{problem}

This is something philosophical. Let us remember that the extra axioms in Theorem 3.2 were obtained by ``putting at work'' the compact quantum groups, in connection with various questions in noncommutative geometry and free probability. There are of course many other potential applications of the compact quantum groups, and putting them at work on other topics, with some axiomatics in mind, could perfectly lead, in the long run, to a different philosophy, different axioms, and a different ``ground zero'' result.

Having such an alternative work done would be of course immensely useful.

\end{document}